\documentclass[a4paper,11pt]{amsart}
\usepackage{graphicx}
\usepackage{url, hyperref}
\usepackage{verbatim}
\usepackage{tikz}
\title{Tropical Geometry in Torus Bundles}
\author{Emily Dodwell}

\newtheorem{theorem}{Theorem}[subsection]
\theoremstyle{definition}
\newtheorem{definition}[theorem]{Definition}
\theoremstyle{lemma}
\newtheorem{lemma}[theorem]{Lemma}
\theoremstyle{definition}
\newtheorem{example}[theorem]{Example}
\theoremstyle{proposition}
\newtheorem{proposition}[theorem]{Proposition}
\theoremstyle{definition}
\newtheorem{construction}[theorem]{Construction}

\usepackage{geometry} 
\begin{document}
\maketitle

\begin{abstract}
We formulate and prove an analogue of the balancing condition for the tropicalization of a curve in a torus bundle $X$. We find that the usual balancing condition fails when the bundle is non-trivial and that the failure is captured by the first Chern classes of the line bundles associated to $X$. We discuss a geometric perspective of tropicalization where the weights arise naturally from intersection theory. The relations between divisors on a toric variety bundle put constraints on these weights which leads to the balancing condition.
\end{abstract}

\section{Introduction}

Tropical geometry is a process which turns subvarieties in algebraic tori into combinatorics. This allows complicated geometric or algebraic problems to be transformed into combinatorial problems, which introduces different techniques and perspectives. See Maclagan-Sturmfels~\cite{MS} for an introduction. Tropical geometry has already been used to establish new results in classical algebraic geometry, such as in the study of curves. For example, see 
\cite{CDPR, FJP, JR}.

We note two fundamental aspects of tropicalization: 
\begin{enumerate}
\item If $Z$ is a subvariety of $(\mathbb{C}^*)^n$, the tropicalization of $Z$, denoted $\text{Trop}(Z)$, is a polyhedral complex in $\mathbb{R}^n$.
\item There are naturally defined weights on the faces of $\text{Trop}(Z)$ which satisfy a natural constraint called the \emph{balancing condition}.
\end{enumerate}
The balancing condition states that for each codimension one face of the tropicalization, the weighted sum of the top-dimensional faces coming out of it is zero. Note that these two facts make tropicalization very concrete.

It is important to note that tropicalization depends on embedding a variety into an algebraic torus. However, this is limited in utility by the fact that many varieties do not admit non-constant maps to algebraic tori. It is therefore natural to consider how to construct tropicalization in a more general setting, preserving its relatively concrete nature. Namely, properties (1) and (2) listed above. Such generalisations have been considered, for example by Ulirsch~\cite{Ulirsch} and Gross~\cite{Gross}, using log geometry. We propose a related generalisation, which is perhaps geometrically more elementary but applies to a similar range of cases. 

Consider a torus bundle, given by the fiber product of a collection of line bundles on a fixed base $B$ minus their zero sections. We can think of this as a family of algebraic tori which varies continuously over $B$. Note that by setting $B$ to be a point, we recover the traditional setting. We will first construct tropicalization for a subvariety $Z$ of a torus bundle. It follows from work of Ulirsch~\cite[Theorem 1.1]{Ulirsch} that the tropicalization of $Z$ is always a polyhedral complex. We will then prove an analogue of the balancing condition for the tropicalization of $Z$ in the case where $Z$ is a curve. For an investigation of higher dimensional balancing, see Carocci-Monin-Nabijou~\cite{CMN}. Our principal findings are that:
\begin{enumerate}
\item The usual balancing condition fails when the bundle is non-trivial.
\item The failure of the balancing condition is exactly captured by the Chern classes of the line bundles associated to the torus bundle.
\end{enumerate}

Our generalisation is compatible with Ulirsch's in the following sense. Fix an SNC pair $(X,D)$ with divisor components $D_1,\ldots,D_k$ and assume all non-empty intersections of subsets of these divisors are connected. We can embed $X$ into the toric variety bundle $\mathbb{P}(\mathcal{O}(-D_1) \oplus \ldots \oplus \mathcal{O}(-D_k) \oplus \mathcal{O})$. Under this embedding, the pullback of the horizontal toric boundary is the divisor $D$ on $X$ and the tropicalization we describe here is exactly the tropicalization described by Ulirsch. We see that an SNC pair $(X,D)$ gives a torus bundle on $X$ where our theory applies. Conversely, a toric compactification of a torus bundle gives a log-regular scheme and Ulirsch's theory applies. In particular, this means that our balancing condition gives a balancing condition for tropicalizations of logarithmic curves in general.

In the context of logarithmic geometry, a balancing condition has been formulated by Gross and Siebert~\cite[Section 1.4]{GS}. The balancing condition seems to be more transparent via torus bundles and therefore gives an elementary approach to understand the balancing condition of Gross-Siebert.

\subsection{Summary of the argument}
We will see that when $Z$ is a curve, the tropicalization of $Z$ is an at most one-dimensional polyhedral complex in $\mathbb{R}^n$. Therefore, the balancing condition is a statement about the weighted sum of the edges around each vertex. We will see that these weights are naturally defined and that there are relations between divisors leading to the balancing condition.

\begin{enumerate}
\item We recall in Section~\ref{sec:geomtrop} a geometric perspective of tropicalization where the weights arise naturally from intersection theory.

\item We recall in Section~\ref{sec:balfromgoem} that relations between divisors on a toric variety put constraints on these weights and that this gives the usual balancing condition.

\item In Section~\ref{sec:tropbundles}, we discuss geometric tropicalization for subvarieties of torus bundles.

\item In Section~\ref{sec:guidingobs} we give examples which demonstrate that the relations between divisors on a toric variety bundle depend on the Chern classes of the line bundles.

\item In Section~\ref{sec:divisors} and \ref{sec:balancing}, we discuss the relations between divisors on a toric variety bundle and formulate an analogue of the balancing condition. The main theorem is stated in Theorem~\ref{thm:balancing}.
\end{enumerate}

\subsection{Future directions}

Some possible generalisations include formulating the balancing condition for subvarieties of a torus bundle of any dimension and considering subvarieties with coefficients defined over a non-Archimedean field e.g. the field of Puiseux series. Another possible direction is to formulate and prove enumerative correspondence theorems for bundles such as in Mikhalkin~\cite{Mikhalkin}.

\subsection{Acknowledgements}

This paper was originally written as an MPhil thesis at the University of Cambridge in 2024. I thank my supervisor, Dhruv Ranganathan, for his invaluable input and support. I would also like to thank Mark Gross and Navid Nabijou for their helpful comments. This work was supported by the Martingale Foundation.

\section{Tropicalization for Subvarieties of Algebraic Tori}
\label{sec:trop}
In this section, we will recall tropicalization for subvarieites of algebraic tori via the valuation map. We also recall geometric tropicalization and discuss the usual balancing condition.

\subsection{Tropicalization via valuation}
\label{sec:tropval}

For simplicity, we will first discuss tropicalization for a curve in $C$ in $(\mathbb{C}^*)^2$. In this case, the tropicalization of $C$ is a one-dimensional polyhedral complex in $\mathbb{R}^2$. The tropicalization of $C$ can be found by following the process outlined below. See Gathmann~\cite{Gathmann} for more details.

Let $K$ be the field of \emph{Puiseux series} whose elements are formal power series $a = \sum_{q \in \mathbb{Q}} a_q t^q$ such that the subset $\{q \in \mathbb{Q} : a_q \neq 0\}$ is bounded below and has a finite set of denominators. Define the \emph{valuation} of $a \in K \setminus 0$, denoted val$(a)$, to be the minimum of this subset. In other words, the valuation of $a$ is the exponent of the leading term.

Begin with a curve in $C \subseteq (\mathbb{C}^*)^2$ defined by an equation $$C = \lbrace (z_1,z_2) \in (\mathbb{C}^*)^2 : f(z_1,z_2) := \sum_{i,j \in \mathbb{Z}} a_{ij} z_1^i z_2^j =0 \rbrace $$ where $a_{ij} \in \mathbb{C}$ with finitely many non-zero. Consider the $K$-points of $C$, i.e. consider the curve $C \subseteq (K^*)^2$ given by $$C = \lbrace (z_1,z_2) \in (K^*)^2 : f(z_1,z_2) := \sum_{i,j \in \mathbb{Z}} a_{ij} z_1^i z_2^j =0 \rbrace.$$
We now define a map 
\begin{align*}
\text{Val}:(K^*)^n & \rightarrow \mathbb{R}^n \\
(z_1,\ldots, z_n)  & \mapsto (x_1,\ldots,x_n) := (-\text{val} (z_1),\ldots, -\text{val}(z_n)) 
\end{align*}
which allows us to find the tropicalization of $C$.

\begin{definition}
\label{def:tropvalC}
Let $C$ be a plane algebraic curve in $(K^*)^2$. The \emph{tropicalization} of $C$ is the closure of $\text{Val}(C)$ in $\mathbb{R}^2$.
\end{definition}

For curves in $(\mathbb{C}^*)^2$, the coefficients $a_{ij}$ of $f$ are in $\mathbb{C}$. In this case, the tropicalization of $C$ will be \emph{conical}. Specifically, it is a one-dimensional polyhedral complex with one vertex at the origin. An example of this can be seen in Figure~\ref{fig:tropC}. Note that the definition above also makes sense when the coefficients $a_{ij}$ are in $K$. In this case, the tropicalization can have multiple vertices which are connected by bounded edges. 

When $C$ is a curve in $(K^*)^2$ defined by an equation $$f(z_1,z_2) := \sum_{i,j \in \mathbb{Z}} a_{ij} z_1^i z_2^j =0,$$
we can also define tropicalization using the \emph{tropical semiring} $(\mathbb{R},\oplus,\odot)$. In this semiring, the operations are given by $$ x \oplus y = \max{\lbrace x,y \rbrace} \quad \text{and} \quad x \odot y = x + y .$$
Define the \emph{tropicalization} of the polynomial $f$ to be 
$$
g(x_1,x_2) = \bigoplus_{i,j} -\text{val}(a_{ij}) \odot {x_1}^{\odot i} \odot {x_2}^{\odot j},
$$
where $g$ can be thought of as a polynomial over the tropical semiring. As a polynomial over $\mathbb{R}$, we have
$$g(x_1,x_2) = \max \lbrace ix_1 + jx_2 -\text{val}(a_{ij}) : i,j \in \mathbb{Z}, a_{ij} \neq 0 \rbrace.
$$
By defining $x_i = -\text{val}(z_i)$, it can be seen that for any point $(z_1,z_2)$ on $C$, the maximum is attained at least twice. The set of all points $(x_1,x_2) \in \mathbb{R}^2$ where the maximum is attained at least twice gives the tropicalization of the curve $C$ defined by the equation $f$. This set of points is the corner locus of $g$, i.e. all points where $g$ is not differentiable.

\begin{definition}
\label{def:troppoly}
Let $C$ be a curve in $(K^*)^2$ defined by $f=0$ and let $g$ be the tropicalization of the polynomial $f$. Then, the tropicalization of $C$ is the corner locus of $g$.
\end{definition}

If $Z$ is defined by the vanishing of the polynomials $f_1,\ldots,f_n$ with corresponding tropical polynomials $g_1,\ldots,g_n$, then $$\text{Trop}(Z) \subseteq \bigcap_i \text{ corner locus of } g_i,$$ but the converse inclusion does not hold in general. However, Definition~\ref{def:tropvalC} may be generalised further for a subvariety $Z \subseteq (K^*)^n$ using the valuation map in the same way. 

\begin{definition}
\label{def:tropvalZ}
Let $Z$ be a subvariety of $(K^*)^n$. The \emph{tropicalization} of $Z$ is the closure of $\text{Val}(Z)$ in $\mathbb{R}^n$.
\end{definition}

We will focus on the case where $Z$ is a curve in $(\mathbb{C}^*)^n$. We will later see in Section~\ref{sec:balfromgoem} that the tropicalization of $Z$ comes with naturally defined positive integer weights on the edges. The \emph{balancing condition} states that around the vertex in the tropicalization of $Z$, the weighted sum of the primitive integral vectors along the edges is zero. This is demonstrated in Example~\ref{eg:trop}. For curves in $(K^*)^n$ whose tropicalizations may have multiple vertices, the balancing condition holds locally around each vertex.

More generally, when $Z$ is a subvariety of $(K^*)^n$, the balancing condition holds around each codimension one face. Consider a codimension one face $\tau$ of $\text{Trop}(Z)$ and the top-dimensional faces $\sigma_i$ such that $\tau$ is a face of $\sigma_i$. Now, choose a generator $v_i$ for each $\sigma_i$ such that $\sigma_i$ is generated by $\tau$ and $v_i$. By working in the quotient $\mathbb{R}^k / \langle \tau \rangle$, where $k = \dim(Z)$, we see that the image of $\tau$ is a point and the image of each $\sigma_i$ is a ray generated by the image of $v_i$. Note that this looks like the tropicalization of a curve in $(\mathbb{C}^*)^n$. In other words, around each codimension one face, we can reduce to the case for curves. Due to this and the fact that balancing is local around each codimension one face, we will focus on the balancing condition for a curve $Z$ in $(\mathbb{C}^*)^n$. Similarly, in Section~\ref{sec:bal}, we will focus on the case where $Z$ is a curve in a torus bundle.

Below is an example of tropicalization for a curve in $(\mathbb{C}^*)^2$. We find the tropicalization and show that the balancing condition holds.

\begin{example}
\label{eg:trop}
Consider the curve $C = \lbrace (z_1,z_2) \in (K^*)^2 : z_1^2 + z_2 = 1 \rbrace$. Then $\text{Val}(z_1,z_2) := (x_1,x_2) = (-\text{val}(z_1),-\text{val}(1 - z_1^2))$.
\begin{itemize}
\item If $\text{val}(z_1) > 0$, then $\text{val}(1 - z_1^2) = 0$.

\item If $\text{val}(z_1) < 0$, then $\text{val}(1 - z_1^2) = 2 \text{val}(z_1)$.

\item If $\text{val}(z_1) = 0$, then $\text{val}(1 - z_1^2) \geq 0$.
\end{itemize}
This gives the tropicalization of $C$, which is shown in Figure~\ref{fig:tropC}.

\begin{figure}[h]
\begin{tikzpicture}
\draw[dotted, thick, <->] (-2,0) -- (2,0) node[anchor=north west] {$x_1$};
\draw[dotted, thick, <->] (0,-2) -- (0,2) node[anchor=south east] {$x_2$};
\draw[black, very thick, ->] (0,0) -- (1,2);
\draw[black, very thick, ->] (0,0) -- (0,-2);
\draw[black, very thick, ->] (0,0) -- (-2,0);
\end{tikzpicture}
\caption{Tropicalization of $C = \lbrace (z_1,z_2): z_1^2 + z_2 = 1 \rbrace$.}
\label{fig:tropC}
\end{figure}
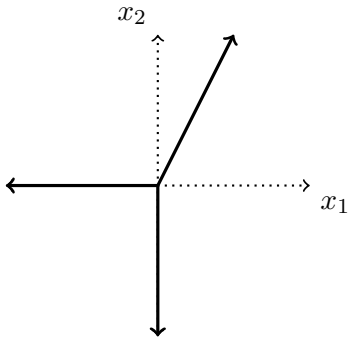

We can also find the tropicalization using the tropical polynomial given by $$g(x_1,x_2) = \max \lbrace 2x_1, x_2, 0 \rbrace.$$ We see that the maximum is attained at least twice when 
$2x_1 = x_2$ and $x_1,x_2 \geq 0$, or when
$x_1 = 0$ and $x_2 \leq 0$, or when
$x_2 = 0$ and $x_1 \leq 0$.
Note that this agrees with the above.

In the case where all coefficients $a_{ij}$ are in $\mathbb{C}$, the tropicalization is conical and we can define the following. 

\begin{itemize}
\item Let $\Delta$ be the convex hull of all points $(i,j)$ for $a_{ij} \neq 0$. Denote the vertices of $\Delta$ by $(p_k,q_k)$ for $k=1,\ldots,n$, labelled in a clockwise direction.
\item Let $u_k = (q_k - q_{k+1}, p_k - p_{k+1})$, where we set $(p_{n+1},q_{n+1}) := (p_1,q_1)$. 
\item Write $u_k = \omega_k \cdot v_k$ where $v_k$ is the primitive integral vector along $u_k$ and $\omega_k$ is a positive integer.
\end{itemize}
Then, the balancing condition states that $\sum \omega_k \cdot v_k = 0$. In other words, the weighted sum of the primitive integral vectors along the edges is $0$. 

For the curve $C$ defined above, we find that the edge pointing downwards in Figure~\ref{fig:tropC} has weight 2, and the other edges have weight 1. This gives the weighted sum
$$(-1,0)+(1,2)+2(0,-1) = (0,0).$$
We see that the tropicalization of $C$ satisfies the balancing condition.
\end{example}

\subsection{Geometric tropicalization}
\label{sec:geomtrop}

We will now discuss an alternative definition of tropicalization due to Tevelev~\cite{Tevelev}. In the case of a curve, this definition will allow us to discuss how the weights are naturally assigned to the rays and how certain relations lead to the balancing condition.

Consider a toric variety $X_\Sigma$ of dimension $n$ and a $k$-dimensional cone $\sigma$ in $\Sigma$. This cone corresponds to an $(n-k)$-dimensional toric variety which we call a \emph{(closed) stratum} of $X_\Sigma$. This toric variety contains an $(n-k)$-dimensional torus which we refer to as a \emph{locally-closed stratum} of $X_\Sigma$.

We see that there is an order-reversing correspondence between cones $\sigma$ in the fan $\Sigma$ and closed strata in $X_\Sigma$. In particular, the rays $\rho$ of the fan $\Sigma$ correspond to divisors $E_\rho$ of $X_\Sigma$ which are irreducible subvarieties of codimension one that are mapped to themselves by the torus. The union of these divisors gives the \emph{toric boundary} $X_\Sigma \setminus (\mathbb{C}^*)^n = \bigcup E_\rho$.

Consider a subvariety $Z$ of $(\mathbb{C}^*)^n$ and its closure $\overline{Z}$ in a toric variety $X_\Sigma$. We say that $\overline{Z}$ meets a stratum of codimension $k$ \emph{(dimensionally) transversely} if 
$$ \dim \left( \overline{Z} \cap \text{stratum of codimension } k \right) = \dim(Z)-k,$$
or if the intersection is empty. In the literature, this is sometimes called a `proper intersection'. The following theorem is due to Tevelev~\cite{Tevelev}.

\begin{theorem}
\label{thm:ztransverse}
Let $Z$ be a subvariety of an algebraic torus $(\mathbb{C}^*)^n$. Then, there exists a toric variety $X_\Sigma$ compactifying $(\mathbb{C}^*)^n$ such that the closure of $Z$ meets the toric boundary transversely.
\end{theorem}

\begin{example}
Consider the toric variety $\mathbb{P}^2$ with homogeneous coordinates $(x_0:x_1:x_2)$. Let $\overline{Z}$ be the curve in $\mathbb{P}^2$ defined by $x_0 = x_1$. This curve meets the point where $x_0 = x_1 = 0$. This intersection is not transverse.

The above theorem states that it is possible to find a toric variety such that the closure of $Z$ meets the toric boundary transversely. In this case, if we instead choose the blow-up of $\mathbb{P}^2$ at the point $x_0 = x_1 = 0$, we see that $\overline{Z}$ meets the toric boundary transversely. This is shown in Figure~\ref{fig:transverse}.

\begin{figure}[h]
\includegraphics[width=250pt]{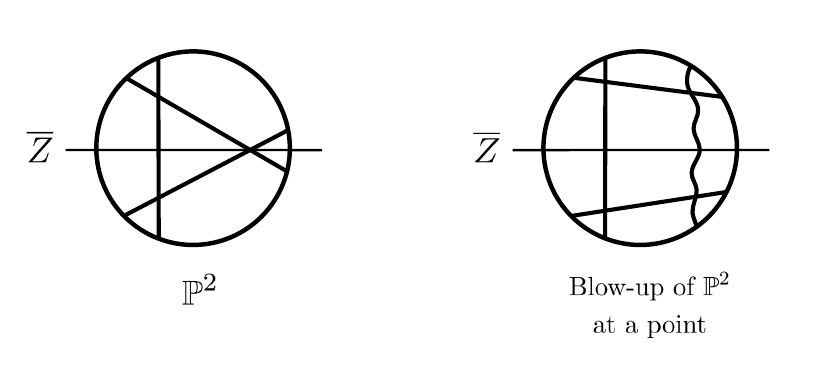}
\caption{Visualisation of $\overline{Z}$ meeting the toric boundary non-transversely and transversely.}
\label{fig:transverse}
\end{figure}

\end{example}

We can now define geometric tropicalization . It follows from Tevelev~\cite{Tevelev} that this definition of tropicalization agrees with Definition~\ref{def:tropvalZ} for a subvariety $Z \subseteq (\mathbb{C}^*)^n$.

\begin{definition}
Let $Z$ be a subvariety of $(\mathbb{C^*})^n$. Let $X_\Sigma$ be a toric variety such that the closure of $Z$ meets the toric boundary transversely. Then, the \emph{tropicalization} of $Z$ is the union of cones $\sigma$ in the fan $\Sigma$ such that the closure of $Z$ meets the strata dual to $\sigma$ in a non-empty subset.
\end{definition}

Consider the case where $Z$ is a curve in $(\mathbb{C^*})^n$. By Theorem~\ref{thm:ztransverse}, we can find a toric variety $X_\Sigma$ such that $\overline{Z}$ meets the toric boundary transversely. We note that the tropicalization is independent of the choice of $X_\Sigma$ and that we can always choose $\Sigma$ to be smooth. We see by counting dimensions that $Z$ will only meet locally-closed strata of codimension $\leq 1$. These correspond to the rays and points of the fan $\Sigma$. 

The definition above then tells us that to find $\text{Trop}(Z)$ for $Z$ a curve, we ask whether $Z$ meets $E_\rho$ for each ray $\rho$. If it does, since the intersection is transverse, we see that $Z$ and $E_\rho$ must meet in a number of points counted with multiplicity. 
\begin{definition}
In the set up above, the \emph{weight} $\omega_\rho$ along a ray $\rho$ in $\text{Trop}(Z)$ is the intersection multiplicity of $Z \cap E_\rho$.
\end{definition}
\noindent
Include this ray in $\text{Trop}(Z)$ and annotate it with the weight. This gives the tropicalization of $Z$.

\subsubsection{Balancing from geometry}
\label{sec:balfromgoem}

We will now discuss the balancing condition for the tropicalization of a curve $Z \subseteq (\mathbb{C}^*)^n$. We have seen that the weights come from intersections with the divisors $E_\rho$. Let $X_\Sigma$ be a smooth toric variety and let $E_1,\ldots,E_k$ be the divisors corresponding to the rays of the fan $\Sigma$. Let $v_1,\ldots, v_k$ be the primitive integral vectors along the rays. Then
$$H^2(X_\Sigma) = \bigoplus_{i=1}^k \mathbb{Z} [E_i] / I $$
where $I$ is the subgroup generated by all expressions $\sum_\rho \langle e_i, v_\rho \rangle \cdot E_\rho$
for $e_i$ in a basis of the cocharacter lattice $M$.

We see that while the divisors $E_\rho$ generate the cohomology of $X_\Sigma$, there are relations. By defining the weight $\omega_\rho := E_\rho \cdot Z$, the relation $\sum \langle e_i, v_\rho \rangle \cdot E_\rho = 0$ gives 
$$\sum \langle e_i, v_\rho \rangle \cdot \omega_\rho = \sum \langle e_i, v_\rho \rangle \cdot E_\rho \cdot Z = 0.$$
For the standard basis $e_i$, this gives a weighted sum on each coordinate which we can instead express as a weighted sum of vectors. This gives $\sum v_\rho \cdot \omega_\rho = 0$, i.e. the weighted sum of the primitive integral vectors along the edges is zero. Note that this is the balancing condition.

\section{Tropicalization for Torus Bundles}
\label{sec:tropbundles}

Let $B$ be a smooth, connected, projective variety. We will now consider a generalisation of the above setting where $Z$ is a subvariety of a torus bundle $X$ over $B$.

\begin{construction}
\label{def:torusbundle}
A \emph{torus bundle} $X$ over $B$ is constructed as follows:
\begin{itemize}

\item Begin with a collection of line bundles $L_1,\ldots,L_n$ on $B$\footnote{Here we view a line bundle as a geometric object rather than a sheaf.}.

\item For each line bundle $L_i$, remove the zero section and denote this $L_i^*$. 

\item Let $X$ be the fibered product $L_1^* \times_B \cdots \times_B L_n^*$.

\end{itemize}
\end{construction}

Note that for a trivialising open $U \subseteq B$, each line bundle $L_i$ on $B$ is locally $U \times \mathbb{C}$ and each $L_i^*$ is locally $U \times \mathbb{C}^*$. Therefore, $X$ is locally $U \times (\mathbb{C}^*)^n$ and $X$ is a bundle over $B$ with fiber $(\mathbb{C}^*)^n$.

\begin{example}
\label{eg:torusbundle}
Consider two line bundles $L_1$ and $L_2$ on $B = \mathbb{P}^1$ and construct a torus bundle $X$ as in Construction~\ref{def:torusbundle}. This can be pictured as shown in Figure~\ref{fig:egbundles1} where above each point in $B$, there is a $(\mathbb{C}^*)^2$ represented by the plane minus the two lines.
\begin{figure}[h]
\includegraphics[width=230pt]{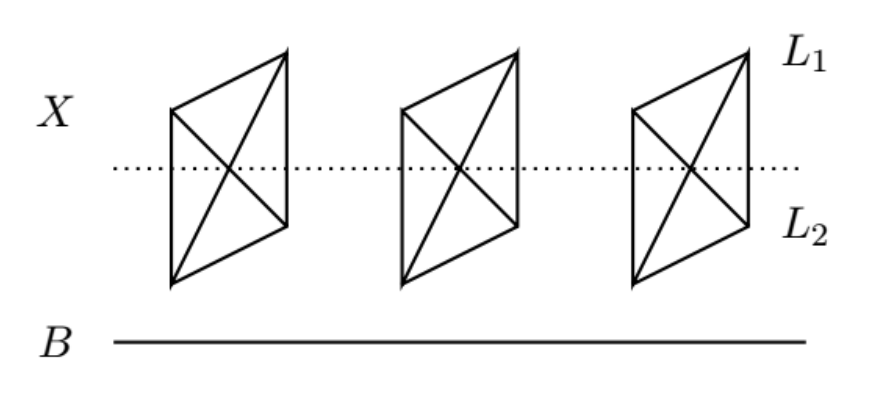}
\caption{A visualisation of a torus bundle.}
\label{fig:egbundles1}
\end{figure}
\end{example}

\begin{construction}
A \emph{toric variety bundle} $Y$ over $B$ is constructed as follows:
\begin{itemize}

\item Begin with a torus bundle $X$ as in Definition~\ref{def:torusbundle} and a fan $\Sigma$ in $\mathbb{R}^n$.

\item For a trivialising open $U \subseteq B$, compactify $X \vert_U = U \times (\mathbb{C}^*)^n$ in $U \times X_\Sigma$, where $X_\Sigma$ is the toric variety associated to the fan $\Sigma$.

\item Glue to get a toric variety bundle $Y$ over $B$.
\end{itemize}
\end{construction}

We see that the toric variety bundle $Y$ locally looks like $U \times X_\Sigma$. Topologically, $Y$ is a fiber bundle over $B$ with fiber $X_\Sigma$:
\begin{align*}
X_\Sigma \;\; \hookrightarrow  \;\; & Y \\
& \downarrow \\
& B
\end{align*}

\begin{example}

Begin with the torus bundle $X$ in Example~\ref{eg:torusbundle} and consider choosing $\Sigma$ to be the fan of $\mathbb{P}^2$. This gives a toric variety bundle $Y$ which can be pictured as shown in Figure~\ref{fig:egbundles2}.

\begin{figure}[h]
\includegraphics[width=230pt]{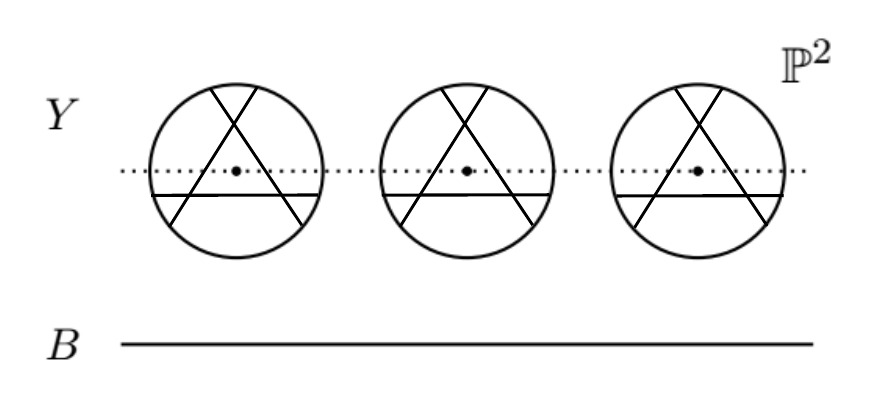}
\caption{A visualisation of a toric variety bundle.}
\label{fig:egbundles2}
\end{figure}

\end{example}

\subsection{Geometric tropicalization for torus bundles}

We will now construct tropicalization for a subvariety $Z$ of a torus bundle $X$. Consider a toric variety bundle $Y$ compactifying $X$ with toric variety $X_\Sigma$ in each fiber. For each ray $\rho$ of $\Sigma$, the divisor $E_\rho$ of $X_\Sigma$ gives a \emph{horizontal divisor} $D_\rho$ of the toric variety bundle $Y$ by taking the union of $E_\rho$ over $B$. In $Y$, we have the \emph{toric boundary} $Y \setminus X = \bigcup D_\rho$. The following theorem follows from Ulirsch~\cite[Theorem 1.2(ii)]{Ulirsch}, which we state here without proof.

\begin{theorem}
\label{thm:ztransversebundle}
Let $X$ be a torus bundle over $B$ and let $Z$ be a subvariety of $X$. Then, there exists a toric variety bundle $Y$ compactifying $X$ such that the closure of $Z$ meets the toric boundary transversely.
\end{theorem}

Ulirsch's proof uses the fact that the bundle can be locally trivialised and reduces to the toric case. It would be interesting to also find a direct proof in the bundle case, by using Gr\"obner theory for bundles as in Tevelev~\cite{Tevelev}.

\begin{definition}
\label{def:geomtropbundle}
Let $X$, $Y$ and $Z$ be as in Theorem~\ref{thm:ztransversebundle}. The \emph{tropicalization} of $Z$ is the union of cones $\sigma$ in the fan of the toric variety bundle $Y$ such that the closure of $Z$ meets the strata dual to $\sigma$ in a non-empty subset.
\end{definition}

In the case where $Z$ is a curve, we see by counting dimensions that $Z$ will only meet locally-closed strata of codimension $\leq 1$ corresponding to the rays and points of the fan $\Sigma$. Therefore, the tropicalization of $Z$ will be a subset of the rays of $\Sigma$ which is at most a one-dimensional polyhedral complex in $\mathbb{R}^n$. Note that in the torus bundle case, the tropicalization of $Z$ may be zero dimensional. For example, consider $B = \mathbb{P}^1$, $L=\mathcal{O}_{\mathbb{P}^1}$ and $Z \subseteq X$ a constant section. This is shown in Figure~\ref{fig:p1bundles}(a).

The definition above then tells us that to find $\text{Trop}(Z)$ for $Z$ a curve, we ask whether $Z$ meets $D_\rho$ for each ray $\rho$. If it does, by Theorem~\ref{thm:ztransversebundle}, we see that $Z$ and $D_\rho$ must meet in a number of points, counted with multiplicity.

\begin{definition}
In the set up above, the \emph{weight} $\omega_\rho$ along a ray $\rho$ in $\text{Trop}(Z)$ is the intersection multiplicity of $Z \cap D_\rho$.
\end{definition}

\noindent Include this ray in $\text{Trop}(Z)$ and annotate it with the weight $\omega_\rho$. This gives the tropicalization of $Z$.

\begin{example}
\label{eg:p1bundles}
Let $Y$ be a $\mathbb{P}^1$-bundle over $\mathbb{P}^1$. There are two horizontal divisors denoted $D_1$ and $D_2$. We will consider tropicalizations of a curve $Z$ in $Y$ for different choices of $Z$. 

If $Z$ misses $D_1$ and $D_2$ as in Figure~\ref{fig:p1bundles}(a), then the tropicalization is just a point. Otherwise, $Z$ could meet the divisors in various other ways leading to different tropicalizations as shown in Figure~\ref{fig:p1bundles}(b) and (c).

\begin{figure}[h]
\includegraphics[width=350pt]{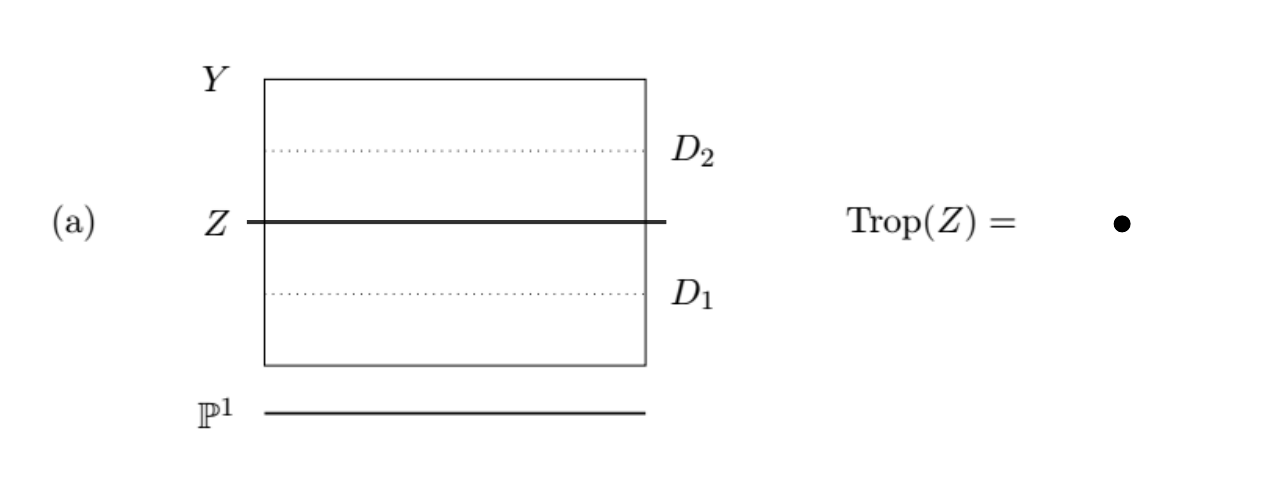}
\end{figure}
\begin{figure}[h]
\includegraphics[width=350pt]{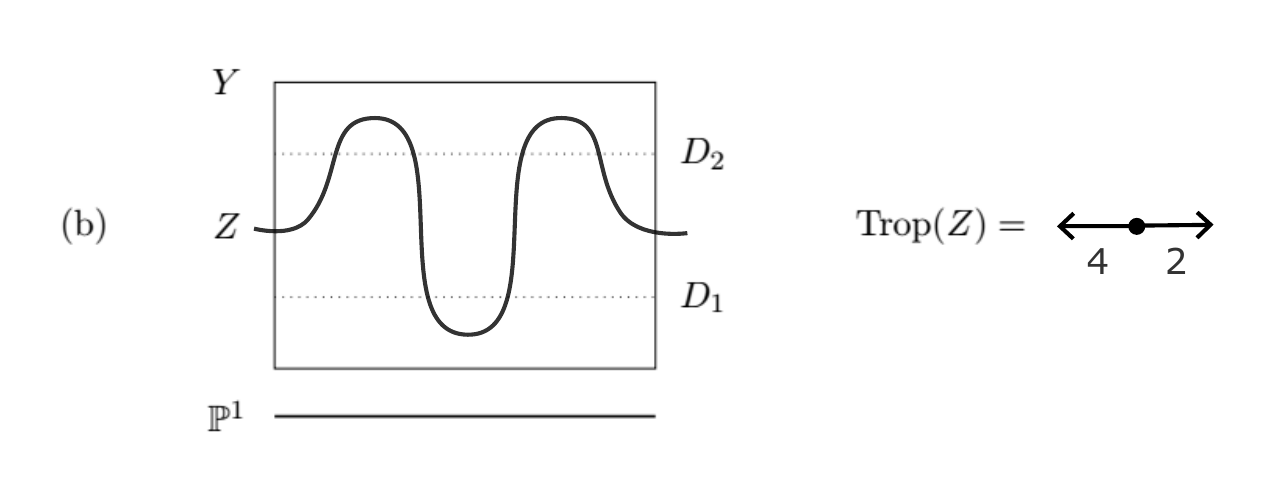}
\end{figure}
\begin{figure}[h]
\includegraphics[width=350pt]{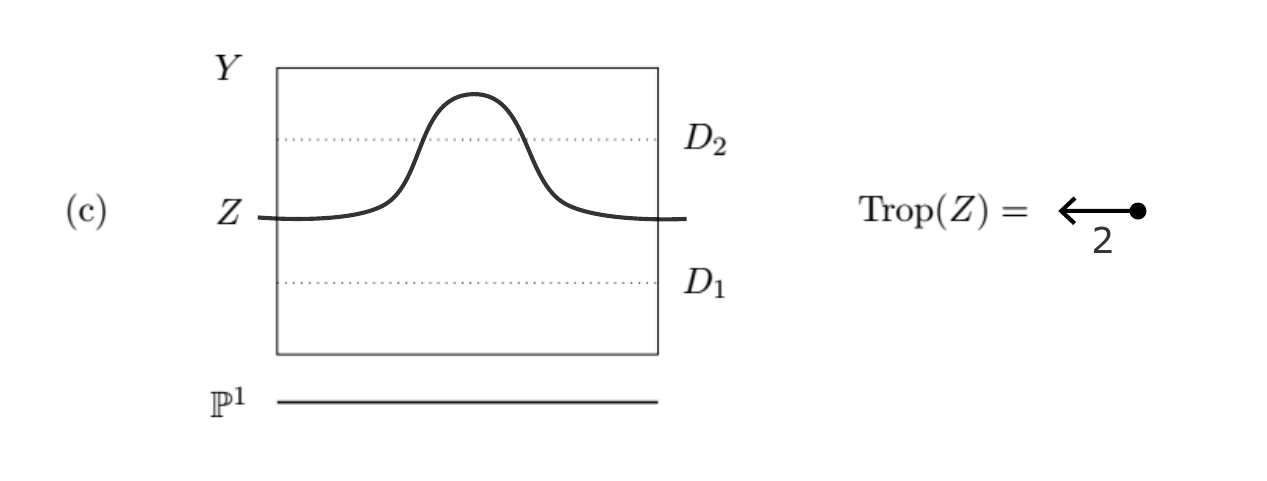}
\caption{Examples of curves in a $\mathbb{P}^1$-bundle over $\mathbb{P}^1$.}
\label{fig:p1bundles}
\end{figure}

\end{example}

\subsection{Tropicalization for torus bundles via valuation}

By setting $B$ to be a point, $Z$ is a subvariety of $(\mathbb{C}^*)^n$ and we recover the original setting. We can also define tropicalization for subvarieties of a torus bundle coordinate-wise using valuation as in Section~\ref{sec:tropval}.

Recall that there are line bundles $L_1, \ldots, L_n$ on $B$. We have an open cover $\lbrace U_\alpha \rbrace$ of $B$ and fiber-preserving isomorphisms 
$$\phi_{i,\alpha}: L_i \vert_{U_\alpha}  \xrightarrow{\sim} U_\alpha \times \mathbb{A}^1.$$ 
These give transition functions 
$$g_{i,\alpha\beta}:= \phi_{i,\alpha} {\phi_{i,\beta}}^{-1} : U_\alpha \cap U_\beta \times \mathbb{A}^1 \rightarrow U_\alpha \cap U_\beta \times \mathbb{A}^1$$
which gives rise to maps 
$$h_{i,\alpha\beta}: U_\alpha \cap U_\beta \rightarrow \mathbb{G}_m,$$
i.e. $g_{i,\alpha\beta}(b,z) = (b, h_{i,\alpha\beta}(b)z)$. These maps are defined over $\mathbb{C}$, so for each $b \in U_\alpha \cap U_\beta$ we have $h_{i,\alpha\beta}(b) \in \mathbb{C}^*$.

We would like to describe a $K$-point $p \in Z(K) \hookrightarrow X(K)$, where $K$ is the field of Puiseux series. In other words, we would like to write $p$ as $(b, z_1, \ldots, z_n)$ for $b \in B, z_i \in L_i(K)$. However, this is only possible on each open set $U_\alpha$. For $p \in U_\alpha \times (K^*)^n$, let $b$ and $z_i$ be such that $\phi_{i,\alpha}(p) = (b, z_i)$. Then define
\begin{align*}
\text{Val}: X(K) & \rightarrow \mathbb{R}^n \\
p & \mapsto (-\text{val}(z_1),\ldots,-\text{val}(z_n)).
\end{align*}
If $p \in U_\alpha \cap U_\beta$, then $p$ may be written in more than one way, say $(b, z_1, \ldots, z_n)$ and $(b, y_1, \ldots, y_n)$. However, the transition functions imply that we have $y_i = h_{i,\alpha\beta}(b)z_i$, where $h_{i,\alpha\beta}(b) \in \mathbb{C}^*$. Therefore, $\text{val}(y_i)=\text{val}(z_i)$ and the map is well-defined.

\begin{definition}
Let $Z$ be a subvariety of a torus bundle $X$. Then, the \emph{tropicalization} of $Z$ is the closure of $\text{Val}(Z)$ in $\mathbb{R}^n$.
\end{definition}

It follows from Ulirsch~\cite{Ulirsch} that this definition of tropicalization agrees with Definition~\ref{def:geomtropbundle}.

\section{The Balancing Condition}
\label{sec:bal}

In this section we will discuss the balancing condition for the tropicalization of a curve $Z$ in a torus bundle $X$. We will see that, as in Section~\ref{sec:balfromgoem}, the balancing condition comes from the divisor relations on a toric variety. We assume that the fan $\Sigma$ of the toric variety bundle $Y$ is smooth and complete. As in the torus case, the tropicalization of $Z$ is independent of the choice of $\Sigma$ and we can always choose $\Sigma$ to be smooth.

\subsection{Guiding observations}
\label{sec:guidingobs}

Consider a toric variety $X_\Sigma$ with the cocharacter lattice $M$. Let $E_\rho$ be the divisor corresponding to the minimal lattice point $v_\rho$ for $\rho \in \Sigma^{(1)}$. If $u$ is any element of $M$, the divisor of the rational function $\chi^u$ on $X_\Sigma$ is $$\text{div}(\chi^u) = \sum \langle u,v_\rho \rangle \cdot E_\rho$$ and this vanishes in $H^2(X_\Sigma)$ since it is the divisor of a rational function. For example, if $X_\Sigma = \mathbb{P}^1$, there are two rays in the fan corresponding to two divisors $E_1$ and $E_2$. Then for $u=1 \in \mathbb{Z} = M$ we have $E_1 - E_2 = 0$ in $H^2(\mathbb{P}^1)$.

These relations also work perfectly well for \emph{trivial} bundles. Consider the trivial bundle $Y = B \times \mathbb{P}^1$. By the K{\"u}nneth theorem, we have
$$ H^2(Y) = H^2(B) \otimes H^0(\mathbb{P}^1) \oplus H^0(B) \otimes H^2(\mathbb{P}^1).$$
We see that the horizontal divisors come from $H^0(B) \otimes H^2(\mathbb{P}^1)$. In this case, the horizontal divisors of $Y$ are given by $D_1 = E_1 \times B$ and $D_2 = E_2 \times B$. Since $E_1 - E_2 = 0$ in $H^2(\mathbb{P}^1)$, we have $$D_1 - D_2 = (E_1 \times B) - (E_2 \times B) = (E_1-E_2) \times B = 0$$ in $H^2(Y)$.

Now, instead consider a \emph{non-trivial} bundle. Let $B = \mathbb{P}^1$, $L_1 = \mathcal{O}(k)$ and let $\Sigma$ be the fan of $\mathbb{P}^1$. Then, the toric variety bundle $Y$ is a $\mathbb{P}^1$-bundle on $\mathbb{P}^1$. In particular, $Y$ is the Hirzebruch surface $\mathbb{P}(\mathcal{O}(k) \oplus \mathcal{O})$. We will later see that $D_1 - D_2 = k \cdot [\text{fiber}]$. 

Notice that the horizontal divisors $D_1$ and $D_2$ of $Y$ are equivalent when the bundle is trivial but differ by $k$ fiber classes when the bundle is non-trivial. We will see that the relations between the horizontal divisors of $Y$ depend on the first Chern classes of the line bundles associated to $Y$. 

\subsection{Divisors on a toric variety bundle}
\label{sec:divisors}

We will now discuss the relations between the horizontal divisors on a toric variety bundle. Since the balancing condition involves algebraic cycles and their intersection numbers, we will use $A_k(\cdot)$ to be the image of the Chow group $CH_k(\cdot)$ of $k$-cycles in $H_{2k}(\cdot)$.

\begin{proposition}
\label{prop:excision}
Let $X$ be an open subscheme of $Y$ and let $D = Y \setminus X$. Let $Y$ have dimension $k$. Then the sequence $$A_{k-1}(D)\rightarrow A_{k-1}(Y) \rightarrow A_{k-1}(X) \rightarrow 0$$ is exact. 
\end{proposition}

\begin{proof}
This follows from Fulton \cite[Proposition 1.8]{Fulton}, and is also known as the excision sequence.
\end{proof}

\begin{lemma}
Let $Y$ be a toric variety bundle of dimension $k$ over a base $B$ of dimension $m$. Then, $A_{k-1}(Y)$ is generated by the pullback of $A_{m-1}(B)$ together with the horizontal torus invariant divisors of the bundle.
\end{lemma}

\begin{proof}
Consider the vector bundle $E$ given by the fiber product $L_1 \times_B \ldots \times_B L_n$ of the line bundles associated to the torus bundle $X \subset Y$. This comes with a map $f:E \rightarrow B$. By \cite[Theorem 3.3]{Fulton}, $$f^*: A_{m-1}(B) \rightarrow A_{k-1}(E)$$ is an isomorphism. By the excision sequence for $X \subset E$, we have a surjective map $A_{k-1}(E) \rightarrow A_{k-1}(X)$. Therefore, there is a surjective map $$A_{m-1}(B) \rightarrow A_{k-1}(X)$$ and we see that $A_{k-1}(X)$ is generated by pullbacks of divisors on $B$.

Now consider the torus bundle $X \subset Y$ and the closed set $D = Y \setminus X$. 
By Proposition~\ref{prop:excision}, we have an exact sequence 
$$A_{k-1}(D)\rightarrow A_{k-1}(Y) \rightarrow A_{k-1}(X) \rightarrow 0$$ where the maps are induced by the inclusions $i: D \hookrightarrow Y$ and $j: X \hookrightarrow Y$. By exactness, we have an isomorphism $$\frac{A_{k-1}(Y)}{\text{Im}(i_*)} = \frac{A_{k-1}(Y)}{A_{k-1}(D)} \cong A_{k-1}(X).$$
Since $D = Y \setminus X = \bigcup D_\rho$ and $\dim(D) = k - 1$, we see that $A_{k-1}(D)$ is generated by the horizontal torus invariant divisors of $Y$. Together with the above, we see that $A_{k-1}(Y)$ is generated by the pullback of $A_{m-1}(B)$ and the horizontal divisors $D_\rho$.
\end{proof}

We have seen that for a toric variety $X_\Sigma$, the relations between the divisors are given by $\sum \langle u,v_\rho \rangle \cdot E_\rho = 0$ for $u \in M$. For a toric variety bundle $Y$, the relations between the horizontal divisors are more complicated as described in the following lemma.

\begin{lemma}
\label{lem:relations}
The relations between the horizontal divisors of a toric variety bundle $\pi:Y \rightarrow B$ are given as follows. For each basis element $e_i$ in the cocharacter lattice $M$ of the torus bundle, there is an associated line bundle $L_i$ and therefore an associated Chern class $c(L_i)$ in the Picard group of $B$. The relation is then: $$\sum_\rho \langle e_i \cdot v_\rho \rangle \cdot D_\rho = \pi^* c_1(L_i),$$ where $v_\rho$ is the primitive integral vector along the edge corresponding to $D_\rho$.
\end{lemma}

\begin{proof}
This follows from Sankaran–Uma~\cite[Theorem 1.2]{SU}.
\end{proof}



It follows from the previous two lemmas that $A_{k-1}(Y)$ is generated by the pullback of $A_{m-1}(B)$ and the horizontal divisors $D_\rho$ with the relations $\sum_\rho \langle e_i \cdot v_\rho \rangle \cdot D_\rho = \pi^* c_1(L_i)$ for all $i$. This gives a description of the divisors on a toric variety bundle, similar to our discussion in Section~\ref{sec:balfromgoem} of divisors on a toric variety. We will now see that this leads to the balancing condition in the torus bundle case in a similar way.

\subsection{The balancing condition}
\label{sec:balancing}

Let $Z$ be a curve in a torus bundle $X$ over $B$, and let $Y$ be a toric variety bundle compactifying $X$ such that the closure of $Z$ meets the toric boundary transversely. Consider the tropicalization of $Z$ which is (at most) a 1-dimensional polyhedral complex in $\mathbb{R}^n$. There are naturally defined weights $\omega_\rho$ on each edge. 

We see that the curve $Z$ determines a function 
\begin{align*}
\Sigma^{(1)} & \rightarrow \mathbb{N} \\
\rho & \mapsto \omega_\rho
\end{align*}
which assigns a weight to each ray in the fan $\Sigma$. However, this function is not arbitrary i.e. the weights are subject to constraints given by $n$ equations. This collection of equations is the balancing condition for the tropicalization of $Z$.

The balancing condition states that the sum over $\rho$ of the $i$th coordinate of $v_\rho$ multiplied by the weight is equal to a number given by the first Chern class of $L_i$. In other words, the weighted sum of the edges is equal to a vector given by the first Chern classes of the $n$ line bundles on $B$.

Define $\beta := \pi_*[Z] \in H_2(B)$. For example, consider $B=\mathbb{P}^1$, then $\beta$ is an integer corresponding to the intersection number of $Z$ with a generic fiber of $\pi$. Recall that the weight $\omega_\rho$ along a ray $\rho$ in $\text{Trop}(Z)$ is defined to be the intersection multiplicity of $Z \cap D_\rho$. We may now state the main theorem.

\begin{theorem}
\label{thm:balancing}
Let $Z$ be a curve in a torus bundle $X$ over $B$ and let $\beta := \pi_*[Z] \in H_2(B)$. Then, the tropicalization of $Z$ satisfies the balancing condition: 
$$\sum_\rho \langle e_i \cdot v_\rho \rangle \cdot \omega_\rho = c_1(L_i)\cdot \beta \quad \forall i.$$
\end{theorem}


\begin{proof}
From Lemma~\ref{lem:relations}, we have the relation
$$\sum_\rho \langle e_i \cdot v_\rho \rangle \cdot D_\rho = \pi^* c_1(L_i).$$
This gives
$$\sum_\rho \langle e_i \cdot v_\rho \rangle \cdot D_\rho \cdot [Z] = \pi^* c_1(L_i)  \cdot [Z].$$
On the left hand side we use the fact that $\omega_\rho = D_\rho \cdot [Z]$ and on the right hand side we apply the projection formula. This gives
\begin{align*}
\sum_\rho \langle e_i \cdot v_\rho \rangle \cdot \omega_\rho  &= c_1(L_i)\cdot \pi_*[Z]. \qedhere
\end{align*}
\end{proof}

\end{document}